\documentclass[a4paper,11pt]{amsart}
\usepackage{graphicx}
\usepackage{mathptmx}
\usepackage{amsmath}
\usepackage{amssymb}
\usepackage{enumitem}
\usepackage{xcolor}
\usepackage{xparse}
\NewDocumentCommand{\eulerian}{omm}
 {%
  \genfrac<>{0pt}{}{#2}{#3}%
  \IfValueT{#1}{_{\!#1}}%
 }

 \newmuskip\pFqmuskip

\newcommand*\pFq[6][8]{%
  \begingroup 
  \pFqmuskip=#1mu\relax
  \mathchardef\normalcomma=\mathcode`,
  \mathcode`\,=\string"8000
  \begingroup\lccode`\~=`\,
  \lowercase{\endgroup\let~}\pFqcomma
  {}_{#2}F_{#3}{\left(\genfrac..{0pt}{}{#4}{#5}\bigg|#6\right)}%
  \endgroup
}
\newcommand{\pFqcomma}{{\normalcomma}\mskip\pFqmuskip}

\newtheorem{theorem}{Theorem}

\newtheorem{corollary}[theorem]{Corollary}
\newtheorem{proposition}[theorem]{Proposition}

\begin{document}

\title[Some identities involving degenerate $r$-Stirling numbers]{Some identities involving degenerate $r$-Stirling numbers}

\author{Taekyun  Kim}
\address{Department of Mathematics, Kwangwoon University, Seoul 139-701, Republic of Korea}
\email{tkkim@kw.ac.kr}

\author{DAE SAN KIM}
\address{Department of Mathematics, Sogang University, Seoul 121-742, Republic of Korea}
\email{dskim@sogang.ac.kr}


\maketitle

\begin{abstract}
Recently, Kim-Kim studied the unsigned degenerate $r$-Stirling number of the first kind and the degenerate $r$-Stirling number of the second kind, respectively of which are the degenerate versions of the unsigned $r$-Stirling numbers of the first kind and those of the $r$-Stirling numbers of the second kind. The aim of this paper is to derive some identities involving such special numbers from the inverse relations for them. 
\end{abstract}

\section{Introduction}
Carlitz initiated a study of degenerate versions of some special polynomials and numbers, namely the degenerate Bernoulli and Euler polynomials and numbers (see [3]). In recent years, we have witnessed that some mathematicians have explored various degenerate versions of many special polynomials and numbers by using various tools like combinatorial methods, generating functions, differential equations, umbral calculus techniques, $p$-adic analysis, special functions, operator theory, probability theory, and analytic number theory. These degenerate versions include the degenerate Stirling numbers of the first and second kinds, which appear very frequently when we study degenerate versions of some special numbers and polynomials (see [5-11]). \par
The unsigned $r$-Stirling number of the first kind ${n \brack k}_{r}$ is the number of permutations of the set $[n]=\{1,2,3,\dots,n\}$ with exactly $k$ disjoint cycles in such a way that the numbers $1,2,\dots,r$ are in distinct cycles, while the $r$-Stirling number of the second kind ${n \brace k}_{r}$ counts the number of partitions of the set $[n]$ into $k$ non-empty disjoint subsets in such a way that the numbers $1,2,\dots,r$ are in distinct subsets. We remark that Border [2] studied the combinatorial and algebraic properties of the $r$-Stirling numbers. \par
The unsigned degenerate $r$-Stirling numbers of the first kind ${n \brack k}_{r,\lambda}$ are degenerate versions of the unsigned $r$-Stirling numbers of the first kind ${n \brack k}_{r}$ and the degenerate $r$-Stirling numbers of the second kind ${n \brace k}_{r,\lambda}$ are degenerate versions of the $r$-Stirling numbers of the second kind ${n \brace k}_{r}$. They can be viewed also as natural extensions of the unsigned degenerate Stirling numbers of the first kind ${n \brack k}_{\lambda}$ and the degenerate Stirling numbers of the second kind ${n \brace k}_{\lambda}$ which were introduced earlier (see [6,10,11]). The aim of this paper is to derive from the inverse relations for the degenerate $r$-Stirling numbers some identities involving such numbers and some special numbers which are given by the evaluations at $r$ of the fully degenerate Bernoulli polynomials, the degenerate two variable Fubini polynomials, the degenerate Euler polynomials and the degenerate poly-Bernoulli polynomials. \par
The outline of this paper is as follows. In Section 1, we recall the degenerate exponential functions, the degenerate logarithms, the degenerate Bernoulli polynomials, the fully degenerate Bernoulli polynomials, the degenerate harmonic numbers, the unsigned degenerate Stirling numbers of the first kind, the degenerate Stirling numbers of the second kind, the unsigned degenerate $r$-Stirling numbers of the first kind, and the degenerate $r$-Stirling numbers of the second kind. In Section 2, we derive the orthogonality relations between the unsigned degenerate $r$-Stirling numbers of the first kind and the degenerate $r$-Stirling numbers of the second kind, and thereby get their inverse relations. By invoking these inverse relations we show some identities involving such degenerate $r$-Stirling numbers and some special numbers which are given by the evaluations at $r$ of aforementioned special polynomials. \par
For any $\lambda\in\mathbb{R}$, the degenerate falling factorial sequence is defined by 
\begin{equation}
(x)_{0,\lambda}=1,\quad (x)_{n,\lambda}=x(x-\lambda)(x-2\lambda)\cdots(x-(n-1)\lambda),\quad (n\ge 1).\label{1}
\end{equation}
Note that $\displaystyle\lim_{\lambda\rightarrow 0}(x)_{n,\lambda}=x^{n},\displaystyle$ (see [7-9]). The degenerate exponential functions are defined by 
\begin{equation}
e_{\lambda}^{x}(t)=\sum_{n=0}^{\infty}(x)_{n,\lambda}\frac{t^{n}}{n!},\quad(\mathrm{see}\ [6,10,11]).\label{2}
\end{equation}
For $x=1$, we let $e_{\lambda}(t)=e_{\lambda}^{1}(t)$. \par 
Let $\log_{\lambda}t$ be the degenerate logarithm which is the compositional inverse of $e_{\lambda}(t)$. Then we have 
\begin{equation}
\log_{\lambda}(1+t)=\sum_{n=1}^{\infty}\frac{\lambda^{n-1}(1)_{n,1/\lambda}}{n!}t^{n},\quad (\mathrm{see}\ [6]).\label{3}	
\end{equation}
Note that $e_{\lambda}\big(\log_{\lambda}(1+t)\big)=\log_{\lambda}\big(e_{\lambda}(1+t)\big)=1+t$. \par 
Carlitz considered the degenerate Bernoulli polynomials given by 
\begin{equation}
\frac{t}{e_{\lambda}(t)-1}e_{\lambda}^{x}(t)=\sum_{n=0}^{\infty}B_{n,\lambda}(x)\frac{t^{n}}{n!},\quad (\mathrm{see}\ [3]). \label{4}	
\end{equation}
Note that $\displaystyle\lim_{\lambda\rightarrow 0}B_{n,\lambda}(x)=B_{n}(x)\displaystyle$, where $B_{n}(x)$ are the ordinary Bernoulli polynomials given by
\begin{displaymath}
	\frac{t}{e^{t}-1}e^{xt}=\sum_{n=0}^{\infty}B_{n}(x)\frac{t^{n}}{n!},\quad (\mathrm{see}\ [1-13]).
\end{displaymath}
Kim-Kim considered the fully degenerate Bernoulli polynomials $\beta_{n,\lambda}(x)$ which are given by 
\begin{equation}
\frac{\log(1+\lambda t)}{\lambda(e_{\lambda}(t)-1)}e_{\lambda}^{x}(t)=\sum_{n=0}^{\infty}\beta_{n,\lambda}(x)\frac{t^{n}}{n!},\quad (\mathrm{see}\ [9]).\label{5}
\end{equation}
Note that $\displaystyle \lim_{\lambda\rightarrow 0}\beta_{n,\lambda}(x)=B_{n}(x),\ (n\ge 0)\displaystyle$. \par 
The degenerate harmonic numbers are defined by Kim-Kim as
\begin{equation}
H_{0,\lambda}=1,\quad H_{n,\lambda}=\sum_{k=1}^{n}\frac{1}{\lambda}\binom{\lambda}{k}(-1)^{k-1},\quad (n\ge 1),\quad (\mathrm{see}\ [8]).\label{6}
\end{equation}
Note that $\displaystyle\lim_{\lambda\rightarrow 0}H_{n,\lambda}=H_{n}=1+\frac{1}{2}+\cdots+\frac{1}{n}\displaystyle$. \par 
From \eqref{6}, we have 
\begin{equation}
-\frac{1}{1-t}\log_{\lambda}(1-t)=\sum_{n=1}^{\infty}H_{n,\lambda}t^{n},\quad (\mathrm{see}\ [8]).\label{7}
\end{equation}
The degenerate Stirling numbers of the first kind are defined by 
\begin{equation}
(x)_{n}=\sum_{k=0}^{n}S_{1,\lambda}(n,k)(x)_{k,\lambda},\quad (n\ge 0),\quad (\mathrm{see}\ [6]).\label{8}
\end{equation}
Note that $\displaystyle\lim_{\lambda\rightarrow 0}S_{1,\lambda}(n,k)=S_{1}(n,k)\displaystyle$ are the Stirling numbers of the first kind.
The unsigned degenerate Stirling numbers of the first kind are defined by Kim-Kim as
\begin{equation}
\langle x\rangle_{n}=\sum_{k=0}^{n}{n \brack k}_{\lambda}\langle x\rangle_{k,\lambda},\quad (n\ge 0),\quad (\mathrm{see}\ [6,10]), \label{9}
\end{equation}
where 
\begin{align*}
&\langle x\rangle_{0}=1,\quad \langle x\rangle_{n}=x(x+1)\cdots(x+n-1), \\
&\langle x\rangle_{0,\lambda}=1,\quad \langle x\rangle_{n,\lambda}=x(x+\lambda)\cdots(x+(n-1)\lambda),\quad (n\ge 1). \end{align*}
Note that $\displaystyle{n \brack k}_{\lambda}=(-1)^{n-k}S_{1,\lambda}(n,k),\ (n,k)\ge 0\displaystyle$.\par 
As the inversion formula of \eqref{8}, the degenerate Stirling numbers of the second kind are defined by Kim-Kim as
\begin{equation}
(x)_{n,\lambda}=\sum_{k=0}^{n}{n \brace k}_{\lambda}(x)_{k},\quad (n\ge 0),\quad (\mathrm{see}\ [6]).\label{10}	
\end{equation}
 Note that $\displaystyle\lim_{\lambda\rightarrow 0}{n \brace k}_{\lambda}={n \brace k}\displaystyle$ are the Stirling numbers of the second kind which are defined by 
 \begin{equation}
 x^{n}=\sum_{k=0}^{n}{n \brace k}(x)_{k},\quad (n\ge 0),\quad (\mathrm{see}\ [12]). \label{11}
  \end{equation}
From \eqref{9} and \eqref{10}, we note that 
\begin{equation}
\frac{1}{k!}\big(-\log_{\lambda}(1-t)\big)^{k}=\sum_{n=k}^{\infty}{n \brack k}_{\lambda}	\frac{t^{n}}{n!},\quad (k\ge 0), \label{12}
\end{equation}
and 
\begin{equation}
\frac{1}{k!}\big(e_{\lambda}(t)-1\big)^{k}=\sum_{n=k}^{\infty}{n \brace k}_{\lambda}\frac{t^{n}}{n!},\quad (\mathrm{see}\ [6]). \label{13}
\end{equation}
Let $r$ be the nonnegative integer. Then the unsigned degenerate $r$-Stirling numbers of the first kind are defined by Kim-Kim as 
\begin{equation}
\langle x+r\rangle_{n}=\sum_{k=0}^{n}{n+r \brack k+r}_{r,\lambda}\langle x\rangle_{k,\lambda},\quad (n\ge 0),\quad (\mathrm{see}\ [6,10,11]).\label{14}
\end{equation}
From \eqref{14}, we have 
\begin{equation}
\frac{1}{(1-t)^{r}}\frac{1}{k!}\big(-\log_{\lambda}(1-t)\big)^{k}=\sum_{n=k}^{\infty}{n+r \brack k+r}_{r,\lambda}\frac{t^{n}}{n!}. \label{15}	
\end{equation}
Note that $\displaystyle \lim_{\lambda\rightarrow 0}{n+r\brack k+r}_{r,\lambda}={n+r \brack k+r}_{r}\displaystyle$ are the unsigned $r$-Stirling numbers of the first kind which are defined by 
\begin{displaymath}
\frac{1}{(1-t)^{r}}\frac{1}{k!}	\big(-\log(1-t)\big)^{k}=\sum_{n=k}^{\infty}{n+r\brack k+r}_{r}\frac{t^{n}}{n!},\quad (\mathrm{see}\ [10]). 
\end{displaymath}
In view of \eqref{10}, the degenerate $r$-Stirling numbers of the second kind are defined by 
\begin{equation}
(x+r)_{n,\lambda}=\sum_{k=0}^{n}{n+r \brace k+r}_{r,\lambda}(x)_{k},\quad (\mathrm{see}\ [6,10,11]). \label{16}
\end{equation}
By \eqref{16}, we easily get 
\begin{equation}
e_{\lambda}^{r}(t)\frac{1}{k!}\big(e_{\lambda}(t)-1\big)^{k}=\sum_{n=k}^{\infty}{n+r \brace k+r}_{r,\lambda}\frac{t^{n}}{n!},\quad (\mathrm{see}\ [6,10,11]).\label{17}	
\end{equation}

\section{Some identities involving degenerate $r$-Stirling numbers}
In this section, we show orthogonality and inverse relations for the unsigned degenerate $r$-Stirling numbers of the first kind and the degenerate $r$-Stirling numbers of the second kind. Then, by using such inverse relations, we derive some identites involving the degenerate $r$-Stirling numbers and some special numbers which are given by the evaluations at $r$ of several special polynomials.\par
From \eqref{14}, we note that 
\begin{align}
\langle x+r\rangle_{n}&=\sum_{k=0}^{n}{n+r \brack k+r}_{r,\lambda}\langle x\rangle_{k,\lambda}= \sum_{k=0}^{n}{n+r \brack k+r}_{r,\lambda}(-1)^{k}(-x)_{k,\lambda}\label{18}\\
&= \sum_{k=0}^{n}{n+r \brack k+r}_{r,\lambda}(-1)^{k}(-x-r+r)_{k,\lambda} \nonumber \\
&= \sum_{k=0}^{n}{n+r \brack k+r}_{r,\lambda}(-1)^{k}\sum_{j=0}^{k}{k+r \brace j+r}_{r,\lambda}(-x-r)_{j}\nonumber \\
&= \sum_{k=0}^{n}\sum_{j=0}^{k}{n+r \brack k+r}_{r,\lambda}{k+r\brace j+r}_{r,\lambda}(-1)^{k-j}\langle x+r\rangle_{j}
\nonumber \\
&=\sum_{j=0}^{n}\bigg(\sum_{k=j}^{n}{n+r \brack k+r}_{r,\lambda}{k+r \brace j+r}_{r,\lambda}(-1)^{k-j}\bigg)\langle x+r\rangle_{j}.\nonumber
\end{align}
By comparing the coefficients on both sides of \eqref{18}, we obtain the first orthogonality relation of the following theorem. The proof of the second one is similar.
\begin{theorem}
For $n\ge 0$, we have the following orthogonality relations:
\begin{align*}
&\sum_{k=j}^{n}(-1)^{n-k}{n+r \brack k+r}_{r,\lambda}{k+r \brace j+r}_{r,\lambda}=\delta_{n,j},\\
&\sum_{k=j}^{n}(-1)^{k-j}{n+r\brace k+r}_{r,\lambda}{k+r\brack j+r}_{r,\lambda}=\delta_{n,j},
\end{align*}
where $\delta_{n,j}$ is the Kronecker's delta.
\end{theorem}
\begin{theorem}
For $n\ge 0$, we have the following inverse relations:
\begin{displaymath}
a_{n,\lambda}=\sum_{k=0}^{n}(-1)^{n-k}{n+r \brack k+r}_{r,\lambda}b_{k,\lambda}\ \Longleftrightarrow\ b_{k,\lambda}=\sum_{k=0}^{n}{n+r \brace k+r}_{r,\lambda}a_{k,\lambda}.
\end{displaymath}	
\end{theorem}
\begin{proof}
$(\Longrightarrow)$ Assume that 
\begin{displaymath}
a_{n,\lambda}=\sum_{k=0}^{n}(-1)^{n-k}{n+r\brack k+r}_{r,\lambda}b_{k,\lambda}.
\end{displaymath}	
Then we have 
\begin{align*}
\sum_{k=0}^{n}{n+r \brace k+r}_{r,\lambda}a_{k,\lambda}&	= \sum_{k=0}^{n}{n+r \brace k+r}_{r,\lambda}(-1)^{k}\sum_{j=0}^{k}(-1)^{j}{k+r \brack j+r}_{r,\lambda}b_{j,\lambda} \\
&=\sum_{j=0}^{n}b_{j,\lambda}\sum_{k=j}^{n}(-1)^{k-j} {n+r \brace k+r}_{r,\lambda}{k+r \brack j+r}_{r,\lambda}\\
&=b_{n,\lambda}.
\end{align*}
$(\Longleftarrow)$ Let $\displaystyle b_{n,\lambda}=\sum_{k=0}^{n}{n+r \brace k+r}_{r,\lambda}a_{k,\lambda}\displaystyle$. Then we note that 
\begin{align*}
\sum_{k=0}^{n}(-1)^{n-k}{n+r \brack k+r}_{r,\lambda}b_{k,\lambda}&= \sum_{k=0}^{n}(-1)^{n-k}{n+r \brack k+r}_{r,\lambda}\sum_{j=0}^{k}{k+r \brace j+r}_{r,\lambda}a_{j,\lambda} \\
&=\sum_{j=0}^{n}a_{j,\lambda}\sum_{k=j}^{n}(-1)^{n-k}{n+r \brack k+r}_{r,\lambda}{k+r \brace j+r}_{r,\lambda}\\
&=a_{n,\lambda}.	
\end{align*}
\end{proof}
In [9], it is shown that 
\begin{equation}
\beta_{n,\lambda}(r)=\sum_{k=0}^{n}(-1)^{k}{n+r \brace k+r}_{r,\lambda}\frac{k!}{k+1}.\label{19}
\end{equation}
Therefore, by Theorem 2 and \eqref{19}, we obtain the following corollary.
\begin{corollary}
	For $n\ge 0$, we have 
	\begin{displaymath}
		\frac{n!}{n+1}=\sum_{k=0}^{n}(-1)^{k}{n+r \brack k+r}_{r,\lambda}\beta_{k,\lambda}(r).
	\end{displaymath}
\end{corollary}
In [5], the degenerate two variable Fubini polynomials are given by 
\begin{equation}
\frac{1}{1-x(e_{\lambda}(t)-1)}e_{\lambda}^{y}(t)=\sum_{n=0}^{\infty}F_{n,\lambda}(x|y)\frac{t^{n}}{n!}.\label{20}	
\end{equation}
We observe that 
\begin{align}
\frac{1}{1-x(e_{\lambda}(t)-1)}e_{\lambda}^{r}(t)&=\sum_{k=0}^{\infty}x^{k}k!	\frac{1}{k!}\big(e_{\lambda}(t)-1\big)^{k}e_{\lambda}^{r}(t)\label{21} \\
&=\sum_{k=0}^{\infty}x^{k}k!\sum_{n=k}^{\infty}{n+r \brace k+r}_{r,\lambda}\frac{t^{n}}{n!} \nonumber \\
&=\sum_{n=0}^{\infty}\sum_{k=0}^{n}x^{k}k!{n+r\brace k+r}_{r,\lambda}\frac{t^{n}}{n!}.\nonumber
\end{align}
From \eqref{20} and \eqref{21}, we note that 
\begin{equation}
F_{n,\lambda}(x|r)=\sum_{k=0}^{n}x^{k}k!{n+r \brace k+r}_{r,\lambda}.\label{22}	
\end{equation}
Therefore, by Theorem 2 and \eqref{22}, we obtain the following theorem. 
\begin{theorem}
For $n\ge 0$, we have 
\begin{displaymath}
x^{n}n!=\sum_{k=0}^{n}(-1)^{n-k}{n+r \brack k+r}_{r,\lambda}F_{k,\lambda}(x|r). 
\end{displaymath}
\end{theorem}
In [3], Carlitz introduced the degenerate Euler polynomials $\mathcal{E}_{n,\lambda}(x)$ which are given by 
\begin{equation}
\frac{2}{e_{\lambda}(t)+1}e_{\lambda}^{x}(t)=\sum_{n=0}^{\infty}\mathcal{E}_{n,\lambda}(x)\frac{t^{n}}{n!}.\label{23}	
\end{equation}
By \eqref{20}, we get 
\begin{align}
\sum_{n=0}^{\infty}F_{n,\lambda}\big(-\frac{1}{2}\ \big|\ r\big)\frac{t^{n}}{n!}&=\frac{1}{1+\frac{1}{2}(e_{\lambda}(t)-1)}e_{\lambda}^{r}(t) \label{24} \\
	&=\frac{2}{e_{\lambda}(t)+1}e_{\lambda}^{r}(t)=\sum_{n=0}^{\infty}\mathcal{E}_{n,\lambda}(r)\frac{t^{n}}{n!}.\nonumber
\end{align}
Thus, we have 
\begin{displaymath}
\mathcal{E}_{n,\lambda}(r)=F_{n,\lambda}\big(-\frac{1}{2}\ \big|\ r\big)=\sum_{k=0}^{n}\big(-\frac{1}{2}\big)^{k}k!{n+r \brace k+r}_{r,\lambda}.
\end{displaymath}
\begin{theorem}
For $n\ge 0$, we have 
\begin{displaymath}
\mathcal{E}_{n,\lambda}(r)=\sum_{k=0}^{n}\big(-\frac{1}{2}\big)^{k}k!{n+r \brace k+r}_{r,\lambda}.
\end{displaymath}	
\end{theorem}
From Theorem 2 and Theorem 5, we note that 
\begin{equation}
\big(-\frac{1}{2}\big)^{n}n!=\sum_{k=0}^{n}(-1)^{n-k}{n+r \brack k+r}_{r,\lambda}\mathcal{E}_{k,\lambda}(r).\label{25}
\end{equation}
Therefore, by \eqref{25}, we obtain the following theorem.\begin{theorem}
For $n\ge 0$, we have 
\begin{displaymath}
	\frac{n!}{2^{n}}=\sum_{k=0}^{n}(-1)^{k}{n+r \brack k+r}_{r,\lambda}\mathcal{E}_{k,\lambda}(r). 
\end{displaymath}	
\end{theorem}
Recently, Kim-Kim introduced the degenerate polylogarithm function of index $k$ which is given by
\begin{equation}
\mathrm{Li}_{k,\lambda}(x)=\sum_{n=1}^{\infty}\frac{(-\lambda)^{n-1}(1)_{n,1/\lambda}}{(n-1)!n^{k}}x^{n},\quad (k\in\mathbb{Z}),\quad (\mathrm{see}\ [6]).\label{26}
\end{equation}
In [9], the degenerate poly-Bernoulli polynomials of index $k$ are defined by 
\begin{equation}
\frac{\mathrm{Li}_{k,\lambda}(1-e_{\lambda}(-t))}{1-e_{\lambda}(-t)}e_{\lambda}^{-x}(-t)=\sum_{n=0}^{\infty}\beta_{n,\lambda}^{(k)}(x)\frac{t^{n}}{n!}.\label{27}
\end{equation}
For $p\in\mathbb{Z}$, we have 
\begin{equation}
\beta_{n,\lambda}^{(p)}(-r)=(-1)^{n}\sum_{k=0}^{n}\frac{\lambda^{k}(1)_{k+1,1/\lambda}}{(k+1)^{p}}{n+r \brace k+r}_{r,\lambda}.\label{28}	
\end{equation}
From Theorem 2 and \eqref{2}, we have 
\begin{align}
\frac{\lambda^{n}(1)_{n+1,1/\lambda}}{(n+1)^{p}}&=\sum_{k=0}^{n}(-1)^{n-k}{n+r \brack k+r}_{r,\lambda}(-1)^{k}\beta_{k,\lambda}^{(p)}(-r)\label{29} \\
&=\sum_{k=0}^{n}(-1)^{n}{n+r \brack k+r}_{r,\lambda}\beta_{k,\lambda}^{(p)}(-r). \nonumber
\end{align}
Therefore, by \eqref{29}, we obtain the following theorem. 
\begin{theorem}
For $n\ge 0$, we have 
\begin{displaymath}
	\frac{(-\lambda)^{n}(1)_{n+1,1/\lambda}}{(n+1)^{p}}=\sum_{k=0}^{n}{n+r \brack k+r}_{r,\lambda}\beta_{k,\lambda}^{(p)}(-r). 
\end{displaymath}	
\end{theorem}
\section{Further remark}
The degenerate Bernoulli polynomials of the second kind $b_{n,\lambda}(x)$ are defined by 
\begin{equation}
\frac{t}{\log_{\lambda}(1+t)}(1+t)^{x}=\sum_{n=0}^{\infty}b_{n,\lambda}(x)\frac{t^{n}}{n!}.\label{30}
\end{equation}
In view of \eqref{7}, we consider the degenerate hyperharmonic numbers which are defined by 
\begin{equation}
H_{n,\lambda}^{(1)}=H_{n,\lambda},\quad H_{n,\lambda}^{(r)}=\sum_{k=1}^{n}H_{k,\lambda}^{(r-1)},\,\,(n \ge 1, r \ge 2),\quad H_{0,\lambda}^{(r)}=0,\,\,(r \ge 2).\label{31}	
\end{equation}
Assume that $r \ge 2$ and that $-\frac{1}{(1-t)^{r-1}}\log_{\lambda}(1-t)=\sum_{n=1}^{\infty}H_{n,\lambda}^{(r-1)}t^{n}$.
From \eqref{31}, we note that 
\begin{align}
\sum_{n=1}^{\infty}H_{n,\lambda}^{(r)}t^{n}&=\sum_{n=1}^{\infty}\sum_{k=1}^{n}H_{k,\lambda}^{(r-1)}t^{n}=\frac{1}{1-t}\sum_{k=1}^{\infty}H_{k,\lambda}^{(r-1)}t^{k}\label{32} \\
&=\frac{1}{1-t}\bigg(-\frac{1}{(1-t)^{r-1}}\log_{\lambda}(1-t)\bigg)=-\frac{1}{(1-t)^{r}}\log_{\lambda}(1-t). \nonumber	
\end{align}
Therefore, by the induction step in \eqref{32}, we obtain the following theorem.
\begin{proposition}
Let $r$ be a positive integer. Then we have 
\begin{displaymath}
-\frac{1}{(1-t)^{r}}\log_{\lambda}(1-t)=\sum_{n=1}^{\infty}H_{n,\lambda}^{(r)}t^{n}.
\end{displaymath}
\end{proposition}
Now, by using \eqref{30} and Proposition 8, we observe that 
\begin{align}
t&=\frac{-t}{\log_{\lambda}(1-t)}(1-t)^{r}\frac{-\log_{\lambda}(1-t)}{(1-t)^{r}}\label{33}\\
&=\sum_{k=0}^{\infty}\frac{(-1)^{k}}{k!}b_{k,\lambda}(r)t^{k}\sum_{l=0}^{\infty}H_{l,\lambda}^{(r)}t^{l} \nonumber 	\\
&=\sum_{n=0}^{\infty}\bigg(\sum_{k=0}^{n}\frac{(-1)^{k}}{k!}b_{k,\lambda}(r)H_{n-k,\lambda}^{(r)}\bigg)t^{n}.\nonumber
\end{align}
Comparing the coefficients on both sides of \eqref{33}, we have 
\begin{equation}
\sum_{k=0}^{n}\frac{(-1)^{k}}{k!}b_{k,\lambda}(r)H_{n-k,\lambda}^{(r)}=\delta_{n,1}.
\end{equation}

\section{Conclusion}

\vspace{0.1in}

The unsigned degenerate $r$-Stirling numbers of the first kind ${n \brack k}_{r,\lambda}$ are degenerate versions of the unsigned $r$-Stirling numbers of the first kind ${n \brack k}_{r}$ and the degenerate $r$-Stirling numbers of the second kind ${n \brace k}_{r,\lambda}$ are degenerate versions of the $r$-Stirling numbers of the second kind ${n \brace k}_{r}$. We derived from the inverse relations for the degenerate $r$-Stirling numbers some identities involving such numbers and some special numbers which are given by the evaluations at $r$ of the fully degenerate Bernoulli polynomials, the degenerate two variable Fubini polynomials, the degenerate Euler polynomials and the degenerate poly-Bernoulli polynomials. \par
It is one of our future projects to continue to explore degenerate versions of some special numbers and polynomials and their applications not only in mathematics but also in other disciplines like statistics, physics, engineering and social sciences.


\vspace{0.1in}

\begin{thebibliography}{9}
\bibitem{1}
Araci, S. \emph{A new class of Bernoulli polynomials attached to polyexponential functions and related identities,} Adv. Stud. Contemp. Math. (Kyungshang) \textbf{31} (2021), no. 2, 195-204.
\bibitem{2}
Broder, A. Z. \emph{The $r$-Stirling numbers,} Discrete Math. \textbf{49} (1984), no. 3, 241-259.
\bibitem{3}
Carlitz, L. \emph{Degenerate Stirling, Bernoulli and Eulerian numbers,} Utilitas Math. \textbf{15} (1979), 51-88.
\bibitem{4}
Comtet, L. \emph{Advanced combinatorics. The art of finite and infinite expansions.} Revised and enlarged edition. D. Reidel Publishing Co., Dordrecht, 1974. 
\bibitem{5}
Kim, D. S.; Jang, G.-W.; Kwon, H.-I.; Kim, T. \emph{Two variable higher-order degenerate Fubini polynomials,} Proc. Jangjeon Math. Soc. \textbf{21} (2018), no. 1, 5-22.
\bibitem{6}
Kim, D. S.; Kim, T. \emph{A note on a new type of degenerate Bernoulli numbers,} Russ. J. Math. Phys. \textbf{27} (2020), no. 2, 227-235.
\bibitem{7}
Kim, T.; Kim, D. S. \emph{Some identities on truncated polynomials associated with degenerate Bell polynomials,} Russ. J. Math. Phys. \textbf{28} (2021), no. 3, 342-355.
\bibitem{8}
Kim, T.; Kim, D. S. \emph{On some degenerate differential and degenerate difference operator,}  Russ. J. Math. Phys. \textbf{29} (2022), no. 1, 37--47.
\bibitem{9}
Kim, T.; Kim, D. S. \emph{Some formulas for fully degenerate Bernoulli numbers and polynomials,} arxiv(2022).
\bibitem{10}
Kim, T.; Kim, D. S.; Lee, H.; Park, J.-W. \emph{A note on degenerate $r$-Stirling numbers,} J. Inequal. Appl. 2020, Paper No. 225, 12 pp.
\bibitem{11}
Kim, T.; Yao, Y.; Kim, D. S.; Jang, G.-W. \emph{Degenerate $r$-Stirling numbers and $r$-Bell polynomials,} Russ. J. Math. Phys. \textbf{25} (2018), no. 1, 44-58.
\bibitem{12}
 Roman, S. \emph{The umbral calculus,} Pure and Applied Mathematics, 111. Academic Press, Inc. [Harcourt Brace Jovanovich, Publishers], New York, 1984.
\bibitem{13}
Simsek, Y. \emph{Identities and relations related to combinatorial numbers and polynomials,} Proc. Jangjeon Math. Soc. \textbf{20} (2017), no. 1, 127-135.
\end{thebibliography}
\end{document}